\documentclass{article}

\usepackage[utf8]{inputenc}
\usepackage{polski}
\usepackage{mathtools}
\usepackage{amssymb}
\usepackage{enumerate}
\usepackage[english]{babel}
\usepackage{hyperref}

\usepackage[table]{xcolor}
\usepackage{todonotes}
\usepackage{hyperref}

\usepackage[noadjust]{cite}
\usepackage{amssymb} \usepackage{graphicx}
\usepackage{enumerate} \usepackage{amsfonts} \usepackage{amsmath}
\usepackage{stmaryrd} 
\usepackage{xspace}  

\usepackage{amsthm}
\theoremstyle{plain}
\newtheorem{theorem}{Theorem}
\numberwithin{theorem}{section}
\numberwithin{figure}{section}
\numberwithin{table}{section}
\newtheorem{lemma}[theorem]{Lemma}
\newtheorem{corollary}[theorem]{Corollary}

\theoremstyle{definition}
\newtheorem{definition}[theorem]{Definition}

\theoremstyle{remark}

\usetikzlibrary{decorations.pathreplacing}
\usetikzlibrary{shapes,snakes}
\tikzstyle{every node} = [draw, circle, fill = black, minimum size = 4pt, inner sep = 0pt]
\tikzstyle{normal} = [draw=none, fill = none]
\tikzstyle{notestyle} = [rectangle,
    draw=black!30!green,
    fill=orange,
    line width=0.5pt,
    text width = \marginparwidth - 1em - 1.6 ex - 1pt,
    inner sep = 0.8 ex,
    rounded corners=4pt]

\title{Certifying coloring algorithms for graphs without long induced paths\thanks{MK was partially supported by the National Science Centre of Poland under grant number 2013/09/B/ST6/03136 and AP was partially supported by the National Science Centre of Poland under grant number 2012/07/D/ST6/02432.}}
\author{Marcin Kamiński\thanks{Corresponding author: {\tt mjk@mimuw.edu.pl}} \and Anna Pstrucha\thanks{The work was done when the author was a student at University of Warsaw.}}

\date{\small Institute of Computer Science\\University of Warsaw, Poland}

\begin{document}

\maketitle

\begin{abstract}
Let $P_k$ be a path, $C_k$ a cycle on $k$ vertices, and $K_{k,k}$ a complete bipartite graph with $k$ vertices on each side of the bipartition. We prove that (1) for any integers $k, t>0$ and a graph $H$ there are finitely many subgraph minimal graphs with no induced $P_k$ and $K_{t,t}$ that are not $H$-colorable and (2) for any integer $k>4$ there are finitely many subgraph minimal graphs  with no induced $P_k$ that are not $C_{k-2}$-colorable. 

The former generalizes the result of Hell and Huang [\emph{Complexity of coloring graphs without paths and cycles},  Discrete Appl. Math. 216: 211--232 (2017)] and the latter extends a result of  Bruce, Ho\`ang, and Sawada [\emph{A certifying algorithm for 3-colorability of $P_5$-Free Graphs}, ISAAC 2009: 594--604]. Both our results lead to polynomial-time certifying algorithms for the corresponding coloring problems.
\end{abstract}

\begin{section}{Introduction}

We consider finite graphs without loops or multiple edges. We use standard notation and refer the reader to \cite{diestel} for the notions not defined here. 

\begin{paragraph}{Graph coloring.}

Let $H$ be a fixed graph. An \textit{$H$-coloring} of a graph $G$ is a mapping from the vertex set $V(G)$ of $G$ to the vertex set of $H$ (colors) with the restriction that adjacent vertices of $G$ are mapped to adjacent colors. When $H$ is the complete graph on $c$ vertices, $H$-coloring corresponds to the standard $c$-coloring. We will also be interested in the algorithmic problem of finding an $H$-coloring. This problem is known to be solvable in polynomial time when $H$ is bipartite and NP-complete otherwise \cite{H-coloring}.

In a generalization of $H$-coloring every vertex of $G$ comes with a list of admissible colors. An \emph{$H$-list-coloring} of $G$ is an $H$-coloring of $G$ such that every vertex of $G$ receives a color that is on its list. When $H$ is the complete graph on $c$ vertices, the problem again specializes to the standard $c$-list-coloring. 

Coloring problems are notoriously NP-complete but for often admit poly\-no\-mial-time algorithms when structural restrictions are placed on the input graph. These are usually phrased in terms of forbidden (induced) subgraphs. For two graphs $H$ and $G$, we write $H \leq G$ when $H$ is an induced subgraph of a graph $G$ and $H \subseteq G$ when $H$ is a subgraph of $G$. We will say that a graph $G$ is \emph{$H$-free} if $H$ is not an induced subgraph of $G$ and that $G$ is $H$-sbgr-free if $H$ is not a subgraph of $G$. 

For which integers $c\geq3$ and graphs $H$ there exist polynomial-time algorithms for the problem of $c$-coloring $H$-free graphs? Kamiński and Lozin in \cite{cycleNP} proved that for any $c \geq 3$, $c$-coloring is NP-complete in the class of $H$-free graphs whenever $H$ contains a cycle. If $H$ is a forest and has a vertex of degree at least 3, then $c$-coloring $H$-free graphs ($c \geq 3$) is NP-complete due to results of Hoyler \cite{3edgecoloring}. Thus polynomial-time algorithms are only possible when $H$ is a vertex-disjoint union of paths and a natural starting point is the case of one path. 

Let $P_k$ denote a path on $k$ vertices.
The first results on coloring $P_k$-free graphs come from  Woeginger and Sgall \cite{longpaths}. They were strengthen by Huang who proved that 5-coloring $P_6$-free graphs and 4-coloring $P_7$-free graphs are both NP-complete \cite{pathNP}. Ho\`ang et al. proved that $c$-coloring and $c$-list-coloring (for any $c>0$) are solvable in polynomial time for $P_5$-free graphs \cite{polyP5}. Chudnovsky et al. proved that 3-coloring is solvable in polynomial time for $P_7$-free graphs \cite{3col}. Two interesting problems that still remain open are the complexity of 4-coloring $P_6$-free graphs and whether $3$-coloring $P_k$-free graphs is NP-complete for some $k>7$? 

Regarding list coloring, Broersma et al. proved that 3-list-coloring is solvable in polynomial time in $P_6$-free graphs \cite{3list} and Golovach et al. showed that 4-list-coloring is NP-complete for $P_6$-free graphs \cite{listP6}. We refer the reader to \cite{survey} for a comprehensive survey of the area. 
\end{paragraph}

\begin{paragraph}{Certifying algorithms.}

An algorithm is \emph{certifying} if together with each output it also return a \emph{witness}, a simple and easily verifiable
certificate that the particular output is correct. A typical example is an algorithm checking whether a graph is bipartite -- it either returns a bipartition or an odd cycle. In any case, there is a certificate of the output. The notion of certifying algorithm was introduced by Kratsch et al. in \cite{certif} and we refer the reader to the survey \cite{certifying} for more details.

A certifying coloring algorithm would return a desired coloring if one exists and an appropriate witness in case of failure. A subgraph (or an induced subgraph) of the input graph that does not have the desired coloring and is minimal with respect to this property is often a good choice of a witness. (As the odd cycle in the case of 2-coloring.) If there are only finitely such many minimal graphs with respect to that coloring, the requirement that the witness is easily verifiable is certainly satisfied. 

If there exits a polynomial-time algorithm for a given coloring problem and there are also only finitely many minimal graphs not admitting this type of coloring, there exists a polynomial-time certifying algorithm for this problem. If additionally, we know all these minimal graphs (or can at least provide a bound on their size), we can actually construct the certifying algorithm.

There is a recent interest in certifying algorithms for coloring problems in graphs without long induced paths. We call a graph $G$ \textit{c-critical} if it cannot be colored with $c-1$ colors but every proper subgraph of $G$ can. Bruce et al. showed that there are six 4-critical $P_5$-free graphs \cite{bruce} and Chudnovsky at el. showed that there are twenty four 4-critical $P_6$-free graphs \cite{4critP6}. The former result together with \cite{polyP5} and the latter together with \cite{3col} imply that there exist certifying algorithms for the corresponding coloring problems. Hell and Huang proved that the set of $c$-critical $(P_6,C_4)$-free graphs is finite for any $c$ \cite{c4}. Their result together with a coloring algorithm of \cite{GPS} provides a certifying algorithm. 

On the other hand, Ho\`ang et al. proved that the set of $c$-critical $P_5$-free graphs is infinite for $c \geq 5$ \cite{p5infCert} and Chudnovsky et al. in fact showed that for a connected $H$ there are finitely many $4$-critical $H$-free graphs if and only if $H$ is a subgraph of $P_6$.
\end{paragraph}

\begin{paragraph}{Our contribution.}

Our contribution is twofold. First, we prove that for any integers $k, t>0$ and a graph $H$ there exists a finite number of subgraph minimal $(P_k,K_{t,t})$-free graphs that are not $H$-list-colorable. This implies the existence of a polynomial-time certifying algorithm for this coloring problem. Consequently, for any integer $c>0$, there are only finitely many $c$-critical $(P_k, K_{t,t})$-free graphs. This generalizes the result of Hell and Huang on $c$-coloring $(P_6,C_4)$-free graphs \cite{c4}, where $C_4 = K_{2,2}$.

Second, we prove that for any integer $k>4$ all subgraph minimal not $C_{k-2}$-colorable graphs have at most $3k+28$ vertices. We also give a certifying algorithm for this coloring problem. This extends a result of Bruce et al. who showed that there are six 4-critical graphs when $k=5$ \cite{bruce}.
\end{paragraph}
\end{section}

\begin{section}{Coloring $(P_k,K_{t,t})$-free graphs}

We will need some definitions. 

\begin{definition}[$r$-minimal]
Let $r$ be a binary relation and $P$ a property. Graph $G$ is \emph{$r$-minimal} for relation $r$ with respect to $P$ iff
there is no graph $G'$ smaller than $G$ with respect to $r$ and having property $P$.
\end{definition}

\begin{definition}[Quasi-ordering]
A pair $(Q, \leq)$ is a \emph{quasi-ordering} if $\leq$ is reflexive and transitive.
\end{definition}

\begin{definition}[Well-quasi-ordering]
A \emph{well-quasi-ordering} (\emph{wqo}) is a quasi-ordering that is well-founded, which means that
any infinite sequence of elements $q_0,q_1,q_2,\ldots$ from $Q$ contains
a pair $q_i \leq q_j$ for some $i<j$.
\end{definition}
 
\begin{definition}[Labeled graph]
Let $(Q, \leq)$ be a quasi-ordering, $G$ be a graph, and let $f$ be a mapping
from $V(G)$ to $Q$. We call the pair $(G,f)$ \emph{$Q$-labeled graph}.
\end{definition}

If $C$ is a class of graphs,
we denote by $C(Q)$ the class of $Q$-labeled graphs $(G,f)$ such that $G \in C$.
We introduce ordering of labeled graphs.
If $G \leq G'$ then there exists a function $\sigma : V(G) \rightarrow V(G')$
such that $(u,v) \in E(G) \iff (\sigma(u), \sigma(v)) \in E(G')$.
If $G \subseteq G'$ then there exists a function $\sigma : V(G) \rightarrow V(G')$
such that $(u,v) \in E(G) \implies (\sigma(u), \sigma(v)) \in E(G')$.

\begin{definition}[Ordering of labeled graph]
For any two members $(G,f)$ and $(G',f')$ of $C(Q)$, we define $(G,f) \leq_l (G',f')$ if
$G \leq G'$ and $f(v) \leq f'(\sigma(v))$ for all $v \in V(G)$.
Similarly $(G,f) \subseteq_l (G',f')$ if
$G \subseteq G'$ and $f(v) \leq f'(\sigma(v))$ for all $v \in V(G)$.
\end{definition}

$P(S)$ denotes the power set of $S$. We are now ready to state and prove our results. 

\begin{theorem}
\label{main}
Let $H$ be a graph and $k,t>0$ integers. Let $C$ be a class of all $(P_k,K_{t,t})$-free graphs.
Let $(Q,\leq)$ be a quasi-ordering where $Q \subseteq P(V(H))$ and $A \leq B$ iff $B$ is a subset of $A$.
There is a finite number of $\leq_l$-minimal ($\subseteq_l$-minimal) not $H$-list-colorable graphs in $C(Q)$.
\end{theorem}

It is easy to see that the theorem easily specializes to many types of graph coloring: $c$-coloring, $H$-coloring, $c$-list-coloring, or $H$-list-coloring. We will state our algorithmic result for these four problems. 

\begin{corollary}
\label{algo}
Let $H$ be a graph and $k,t>0$ integers. There exists a polynomial certifying algorithm deciding $c$-colorability
($H$-colorability, $c$-list-colorability, $H$-list-colorability) of a $(P_k, K_{t,t})$-free graph.
\end{corollary}

\noindent \emph{Sketch of proof}. To prove the corollary we have to give a polynomial-time coloring algorithm when the input graph is $c$-colorable
($H$-colorability, $c$-list-colorability, $H$-list-colorability). Let $H$, $k$, and $t$ be like in the statement of the Corollary. We will see in the next section (Lemma \ref{l}) that $H$-list-colorable $(P_k, K_{t,t})$-free graphs are $P_l$-sbgr-free, where $l$ is a constant depending on $k$, $t$, and $H$. Excluding $P_l$ as a subgraph is equivalent to excluding $P_l$ as a minor. Robertson and Seymour proved that graphs excluding a planar minor have bounded treewitdh \cite{planar-minor}. Hence, $H$-list-colorable $(P_k, K_{t,t})$ graphs have bounded treewidth and can be colored in any of the coloring models by designing a dynamic programming over a tree decomposition \cite{bounded-twd-Hcol}.

\begin{subsection}{Proof of Theorem \ref{main}}

In this proof we will shorten $H$-list-colorable to just colorable.
Denote the class of all colorable graphs in $C(Q)$ by $S$. Notice that if
$(G,f) \in S$ and $(G',f') \leq_l (G,f)$ or $(G',f') \subseteq_l (G,f)$ then $(G',f') \in S$.
Denote the set of all $\leq_l$-minimal not colorable $Q$-labeled graphs in $C$ by $A$.
Denote the set of all $\subseteq_l$-minimal not colorable $Q$-labeled graphs in $C$ by $A'$.
We want to prove that $A$ and $A'$ are finite.

We need to use the following two lemmas from literature. The first lemma was proved by Ding in \cite{wqo} and the second by Atminas et al. in \cite{ramsey}

\begin{lemma}
\label{wqo}
Let $X$ be a class of all $P_k$-sbgr-free graphs and let $Y$ be a quasi-ordering. If $(Y, \leq)$ is
wqo then $(X(Y), \leq_l)$ is wqo.
\end{lemma}

\begin{lemma}
\label{ramsey}
For every $t$, $q$ and $s$, there is a number $z = Z(s, t, q)$ such that every
graph with a path of length at least $z$ contains either $K_t$ or $K_{q,q}$ or $P_s$ as induced subgraph.
\end{lemma}

Now we are ready to prove Theorem \ref{main}.

\begin{lemma}
\label{l}
There exists a constant $l$ such that all graphs in $S$ are $P_l$-sbgr-free.
\end{lemma}
\begin{proof}
Let $(G,f) \in S$.
Vertices in $G$ that form a clique can only be colored by distinct vertices of $H$
that form a clique in $H$.
This means that $\omega(G) \leq \omega(H)$, so graphs in $S$ are
$K_{\omega(H)+1}$-free. They belong to $C$, so they are also $(P_k,K_{t,t})$-free.
We get the result from Lemma \ref{ramsey}.
\end{proof}

\begin{lemma}
\label{l'}
There exists a constant $l'$ such that all graphs in $A$ are $P_{l'}$-sbgr-free.
\end{lemma}
\begin{proof}
Take a graph in $(G,f) \in A$. There exists $(G',f') \in S$ such that $(G',f') \leq_l (G,f)$.
$(G,f)$ is $\leq_l$-minimal.
If $G'$ has the same vertex set as $G$ then $G$ and $G'$ differ only by labels and we are done.
Otherwise, we can choose $G'$ in such way that $G' = G - \{ v \}$ for some vertex $v \in G$.
Take any path $P$ in $G$. If $P$ does not use vertex $v$ then its length is at most $l-1$
from previous lemma.
Otherwise, vertex $v$ splits the path into two parts, both included in $G'$,
so each part is of length at most $l-1$ from previous lemma. Adding vertex $v$ we obtain that $P$ is
of length at most $2(l-1) + 1 = 2l - 1$. This means that we can take $l' = 2l$.
\end{proof}

\begin{lemma}
\label{Awqo}
Graphs in $A$ are wqo by $\leq_l$ relation.
\end{lemma}
\begin{proof}
From Lemma \ref{l'} graphs in $A$ are $P_{l'}$-sbgr-free. $(Q, \leq)$ is wqo because it is
quasi ordering of a finite set. From Lemma \ref{wqo} class of $Q$-labeled $P_{l'}$-sbgr-free graphs
is wqo by $\leq_l$. $A$ is a subset of this class, so it is also wqo by $\leq_l$ relation.
\end{proof}

\begin{lemma}
$A$ is finite.
\end{lemma}
\begin{proof}
From Lemma \ref{Awqo} graphs in $A$ are wqo by $\leq_l$ relation, so $A$ contains no
infinite antichain. $A$ consists of $\leq_l$-minimal graphs, so $A$ is an antichain.
Thus $A$ cannot be infinite.
\end{proof}

\begin{lemma}
$A'$ is finite.
\end{lemma}
\begin{proof}
Take any graph $(G,f) \in A'$. $(G,f)$ is not colorable, so there exists $(G',f') \in A$ such that
$(G',f') \leq_l (G,f)$. Then also $(G',f') \subseteq_l (G,f)$. Since $(G,f)$ is $\subseteq_l$-minimal
not colorable we have $(G',f') = (G,f)$. This means that $A' \subseteq A$, so $A'$ is finite.
\end{proof}

We have proved that $A$ and $A'$ are finite, which finishes the proof of Theorem \ref{main}.

\end{subsection}

\end{section}

\begin{section}{$C_{k-2}$-coloring $P_k$-free graphs}

In this section we will prove the following two theorems.

\begin{theorem} \label{main1}
Let $k>4$ be an integer. Any $\leq$-minimal $P_k$-free not $C_{k-2}$-colorable graph has at most $3k+28$ vertices.
\end{theorem}

\begin{theorem} \label{main2}
There exists a polynomial  certifying algorithm for $C_{k-2}$-coloring $P_k$-free graphs.
\end{theorem}

Let $G$ be a $P_k$-free graph. First, let us assume that $k$ is even. In this case, $G$ is $C_{k-2}$-colorable
iff $G$ is 2-colorable. If $G$ is bipartite, we can color $G$ by two subsequent vertices
of $C_{k-2}$ in linear time. Otherwise, $G$ contains an odd cycle of length at most $k-1$. We can find this odd cycle and return as a witness of size at most $k-1$.

The case of $k=5$ is the problem of 3-coloring $P_5$-free graphs. We know from \cite{4critP6} that there is a finite
number of 4-critical $P_5$-free graphs and they have at most $16$  vertices.

We can assume that $k$ is odd and $k \geq 7$.
If $G$ is bipartite, then $G$ is clearly $C_{k-2}$-colorable,
because again we can use only two subsequent vertices of $C_{k-2}$ to color $G$. Otherwise, $G$ contains an odd
cycle of length at most $k$. If $G$ contains an odd cycle of length $l \leq k-4$, then $G$ is not
$C_{k-2}$-colorable, since $C_l$ is not $C_{k-2}$-colorable. Hence, any odd cycle of $G$ is of length $k-2$ or $k$. We will consider two cases: $G$ contains $C_k$ and $G$ is $C_k$-free.

\begin{subsection}{$G$ contains an induced $C_k$}

Let $C$ be a fixed cycle of length $k$ in $G$. Let $v_1,\ldots,v_k$ be
vertices of $C$. Let $c(v)$ denote the color of vertex $v$.
We start from describing the neighbourhood of $C$.

\begin{lemma}
Let $v \in N(C)$. There are only three possibilities:
\begin{enumerate}[(i)]
\item $v$ has $2$ neighbours in $C$, $v_i$ and $v_{i+2}$ (we denote $v \in D_i$)
\item $v$ has $2$ neighbours in $C$, $v_i$ and $v_{i+4}$ (we denote $v \in T_i$)
\item $v$ has $3$ neighbours in $C$, $v_i$, $v_{i+2}$ and $v_{i+4}$ (we denote $v \in T_i$ as well)
\end{enumerate}
for some $1 \leq i \leq k$ (see Figure \ref{picturek}).
\end{lemma}

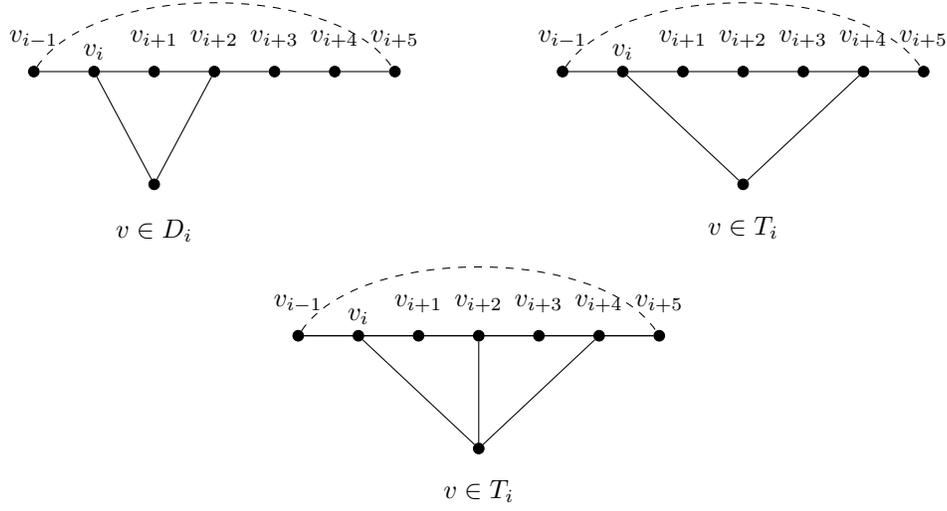
\begin{figure}
\centering
\begin{tikzpicture}
\begin{scope}
    \node (v1) [label=$v_{i-1}$] at (-1*0.8,0) {};
    \node (v2) [label=$v_{i  }$] at (0*0.8,0) {};
    \node (v3) [label=$v_{i+1}$] at (1*0.8,0) {};
    \node (v4) [label=$v_{i+2}$] at (2*0.8,0) {};
    \node (v5) [label=$v_{i+3}$] at (3*0.8,0) {};
    \node (v6) [label=$v_{i+4}$] at (4*0.8,0) {};
    \node (v7) [label=$v_{i+5}$] at (5*0.8,0) {};
    \node (v8) [label=below:$v \in D_i$] at (1*0.8,-1.5) {};
    \draw (v1) -- (v2);
    \draw (v2) -- (v3);
    \draw (v3) -- (v4);
    \draw (v4) -- (v5);
    \draw (v5) -- (v6);
    \draw (v6) -- (v7);
    \draw (v2) -- (v8);
    \draw (v4) -- (v8);
    \draw [dashed] (v1) .. controls (0,1.2) and (4*0.8,1.2)  .. (v7);
\end{scope}
\begin{scope}[xshift=200]
    \node (v1) [label=$v_{i-1}$] at (-1*0.8,0) {};
    \node (v2) [label=$v_{i  }$] at (0*0.8,0) {};
    \node (v3) [label=$v_{i+1}$] at (1*0.8,0) {};
    \node (v4) [label=$v_{i+2}$] at (2*0.8,0) {};
    \node (v5) [label=$v_{i+3}$] at (3*0.8,0) {};
    \node (v6) [label=$v_{i+4}$] at (4*0.8,0) {};
    \node (v7) [label=$v_{i+5}$] at (5*0.8,0) {};
    \node (v8) [label=below:$v \in T_i$] at (2*0.8,-1.5) {};
    \draw (v1) -- (v2);
    \draw (v2) -- (v3);
    \draw (v3) -- (v4);
    \draw (v4) -- (v5);
    \draw (v5) -- (v6);
    \draw (v6) -- (v7);
    \draw (v2) -- (v8);
    \draw (v6) -- (v8);
    \draw [dashed] (v1) .. controls (0,1.2) and (4*0.8,1.2)  .. (v7);
\end{scope}
\begin{scope}[xshift=100,yshift=-100]
    \node (v1) [label=$v_{i-1}$] at (-1*0.8,0) {};
    \node (v2) [label=$v_{i  }$] at (0*0.8,0) {};
    \node (v3) [label=$v_{i+1}$] at (1*0.8,0) {};
    \node (v4) [label=$v_{i+2}$] at (2*0.8,0) {};
    \node (v5) [label=$v_{i+3}$] at (3*0.8,0) {};
    \node (v6) [label=$v_{i+4}$] at (4*0.8,0) {};
    \node (v7) [label=$v_{i+5}$] at (5*0.8,0) {};
    \node (v8) [label=below:$v \in T_i$] at (2*0.8,-1.5) {};
    \draw (v1) -- (v2);
    \draw (v2) -- (v3);
    \draw (v3) -- (v4);
    \draw (v4) -- (v5);
    \draw (v5) -- (v6);
    \draw (v6) -- (v7);
    \draw (v2) -- (v8);
    \draw (v4) -- (v8);
    \draw (v6) -- (v8);
    \draw (v1) -- (v7);
    \draw [dashed] (v1) .. controls (0,1.2) and (4*0.8,1.2)  .. (v7);
\end{scope}
\end{tikzpicture}
\caption{Possible connections of $v \in N(C)$ to the cycle $C$ of length $k$. \label{picturek}}
\end{figure}

\begin{proof}
Suppose $v$ is adjacent to exactly one vertex of $C$, then we obtain $P_k$ in $G$. $v$ cannot have
two consecutive vertices of $C$ as neighbours since this would form a triangle. Suppose $v$ is
connected to some vertex $i$ and $i+t$ where $t$ is odd, $3 \leq t < k - 4$
and $v$ is not connected to $i+1, ..., i+t-1$. In this case we obtain an odd cycle of length smaller
than $k-2$. Notice that for every neighbour $v$ of $C$ there will be such $i$ and $t$ that
$v$ is connected to $i$ and $i+t$, $t$ is odd, $3 \leq t \leq k-2$ and $v$ is not connected to
$i+1,...,i+t-1$. So the only allowed values of $t$ are $k-4$ and $k-2$.
This gives us neighbours in distance 2 or 4 only.
\end{proof}

Let $D = \cup_{i=1}^k D_i$ and $T = \cup_{i=1}^k T_i$.
Let us denote by $S$ the set of vertices $S = C \cup D \cup T$.

\begin{lemma}
Let $v \in D_i$. $N(v) \subseteq C \cup T \cup D_{i\pm1} \cup D_{i\pm3}$. 
\end{lemma}

\begin{proof}
$v$ cannot have neighbours in $G-S$ since this would form $P_k$.
$D_i$ is a stable set since otherwise we would have
a triangle. Vertices from $D_i$ cannot be connected to $D_{i+2}$ since this would form a triangle.
Vertices from $D_i$ cannot be connected to vertices of $D_{i+t}$ for
$4 \leq t \leq k - 4$ since this would give an odd cycle of length at most $k-4$.
\end{proof}

\begin{lemma} \label{always}
If $G$ has no vertices of type $T$ and no edge between $D_i$ and $D_{i+3}$
then $G$ is $C_{k-2}$-colorable.
\end{lemma}

\begin{proof}
The only edges of $G$ belong to $C$, connect a vertex from $D_i$ to vertices $i$ and $i+2$
of $C$ and connect vertices of type $D_i$ to vertices of type $D_{i+1}$.
In order to $C_{k-2}$-color $G$ we can choose arbitrary coloring of $C$ and
then color all vertices of $D_i$ by color of vertex $v_{i+1}$.
\end{proof}

From now on we assume that there exists a vertex of type $T$ or an edge between some vertex
in $D_i$ and $D_{i+3}$. In any $C_{k-2}$-coloring of $C$ there are exactly 2 colors
which are repeated twice and they are subsequent vertices of $C_{k-2}$.

\begin{lemma} \label{H0}
There exists a connected graph $H_0$, $H_0 \leq G$, $H_0$ consists of $C$ and at most $2$ other
vertices such that $H_0$ forbids all but $5$ $C_{k-2}$-colorings of $C$ in $G$.
\end{lemma}

\begin{proof}
Vertices from $D_i$ can be connected to vertices of $D_{i+3}$ only if a color is repeated
in coloring of $C$ in the interval $[i, i+5]$. Moreover, if we have a vertex of type $T_i$
a color must be repeated in coloring of $C$ in the interval $[i,i+4]$.
\end{proof}

We will now aim to prove Lemma \ref{fixedCol}.
Suppose we have a fixed $C_{k-2}$-coloring of $C$ in $G$. We fix $c(v_i) = i$
for $i \leq k-2$, $c(v_{k-1}) = 1$, $c(v_k) = 2$.

\begin{lemma}
If $T_i \neq \emptyset$ for $i \notin [k-3,k]$ or there exists an edge from $D_i$ to
$D_{i+3}$ for $i \notin [k-4,k]$ we cannot extend $C_{k-2}$-coloring of $C$ and we have a witness consisting
of $C$ and at most $2$ vertices.
\end{lemma}

\begin{proof}
This follows from the proof of Lemma \ref{H0}.
If $T_i \neq \emptyset$ for $i \notin [k-3,k]$ then the witness is a graph consisting of $C$
and any vertex from this $T_i$.
If there exists an edge $(u,v)$ from $D_i$ to $D_{i+3}$ and $i \notin [k-4,k]$ then
the witness is a graph consisting of $C$, $u$ and $v$.
\end{proof}

\begin{lemma} \label{typet}
Either all vertices of type $T$ and of type $D_i$ for $i \neq k-1,k$ have
exactly one possible color or we cannot extend $C_{k-2}$-coloring of $C$
to $T$ and $\cup_{i=1}^{k-2} D_i$
and there is a witness consisting of $C$ and at most $2$ vertices.
\end{lemma}

\begin{proof}
Vertices of $T_{k-3}$, $T_{k-2}$, $T_{k-1}$, $T_k$ must have colors respectively $k-2$, $1$, $2$, $3$.
All vertices in $D_i$ for $i \neq k-1,k$ must have color $c(v_{i+1})$.
If any edge between these vertices causes a conflict, we cannot extend the coloring of $C$.
A witness consists of $C$ and endpoints of this edge.
\end{proof}

\begin{lemma} \label{typed}
Either we can extend a $C_{k-2}$-coloring of $C$, $T$ and $\cup_{i=1}^{k-2} D_i$ to a $C_{k-2}$-coloring of $S$
or there exists a witness consisting of $C$ and at most $4$ vertices.
\end{lemma}

\begin{proof}
We need to color vertices of $D_{k-1}$ and $D_k$.
Vertices in $D_{k-1}$ can have color $k-2$ or $2$, vertices in $D_k$ can have
color $1$ or $3$. We take a vertex $v$ from $D_{k-1}$ or $D_k$. It may have neighbours
which cause a conflict of colors in $v$. Then we have a witness consisting of cycle $C$,
$v$ and 2 neighbours. Otherwise, each vertex is colored or still has two possible colors.
If two endpoints of edge $(u,v)$ connecting $D_{k-1}$ and $D_k$ have been colored by
$k-2$ and $3$ respectively, then we have a witness consisting of cycle $C$, $u$, $v$
one neighbour of $u$ and one neighbour of $v$ causing their colors.
If we did not find any conflicts, we can color all remaining vertices of $D_{k-1}$ by $2$
and all remaining vertices of $D_k$ by $1$, completing the coloring of $S$.
\end{proof}

\begin{lemma} \label{g-s}
Either we can extend a $C_{k-2}$-coloring of $C$, $T$ and $\cup_{i=1}^{k-2} D_i$ to a coloring of 
$G - \cup_{i=k-1}^k D_i$
or there exists a witness consisting of $C$ and at most $8$ vertices.
\end{lemma}

\begin{proof}
Now we consider vertices in $G - S$. Their neighbours in $S$ are of type $T$ only.
We do similar structural analysis as in \cite{3col}.
Let $M$ be a connected component of $G - S$. Suppose $M$ consists of
at least $2$ vertices. $M$ is bipartite
(otherwise we would have cycle $C'$ of length $k-2$ or $k$ in $G-S$ connected by some path $p_1..p_s$
to $C$, we would have induced $P_k$ $c'_3c'_2c'_1p_1..p_sc_i...c_{i+k-5}$). If $xyz$ is a path in $M$ then
$x$ and $z$ have the same neighbourhood in $S$ (otherwise let $u$ of type $T$ be a neighbour of $x$
not adjacent to $z$, $zyxuc_i...c_{i+k-5}$ is an induced $P_k$). So vertices in each of the two stable
sets in $M$ are adjacent to the same vertices in $S$. This means that we can color $M$ iff we can
color one edge of $M$ (all vertices in a stable set can be colored by the same color). We call this
edge $M'$. Otherwise, $M$ consists of one vertex and we define $M'=M$.
Now take an induced subgraph $H$ of $G$. $H$ consists of $C$, $M'$
and one vertex $v$ from each $T_i$ for each vertex $u$ of $M'$ if $u$ is adjacent to $v$.
If the coloring of $C$ extends to a coloring of $H$ and we have colored $S$ then we can
color $S \cup M$. If it holds for every component $M$, we have colored $G$. Otherwise,
the coloring of $C$ does not extend to a coloring of some $H$. $H$ is a witness which
consists of $C$ and at most 8 other vertices (it is enough to take neighbours from 3 sets $T_i$).
\end{proof}

\begin{lemma} \label{fixedCol}
A fixed $C_{k-2}$-coloring of $C$ either extends to a $C_{k-2}$-coloring of $G$ or there exists a connected witness
consisting of $C$ and at most 8 other vertices.
\end{lemma}

\begin{proof}
This follows form Lemmas \ref{typet}, \ref{typed}, \ref{g-s} and the fact that there are no edges
between $G-S$ and $D$.
\end{proof}

\begin{lemma} \label{Ck}
Either we can $C_{k-2}$-color $G$ or $G$ contains a not $C_{k-2}$-colorable induced subgraph of size
at most $k+42$ containing $C_k$.
\end{lemma}

\begin{proof}
$G$ can be $C_{k-2}$-colorable from Lemma \ref{always}. Otherwise, we can have two cases.
In the first case for one of the $5$ $C_{k-2}$-colorings of $C$ we have found an extension
to the coloring of $G$. In this case we can color $G$. In the second case
for each of these coloring types we have a witness consisting of $C$ and at most
$8$ other vertices from Lemma \ref{fixedCol}. We can combine the witnesses together with $H_0$
from Lemma \ref{H0} to obtain
a witness that $G$ is not $C_{k-2}$-colorable. This witness has at most $k+2+ 5*8 = k+42$ vertices and
contains $C$.
\end{proof}

\end{subsection}

\begin{subsection}{$G$ is $C_k$-free}

Let $C$ be a fixed cycle of length $k-2$ in $G$. Let $v_1,\ldots,v_{k-2}$ be
vertices of $C$. Let $c(v)$ denote the color of vertex $v$.
There is exactly one possible $C_{k-2}$-coloring of $C$.
We start from describing the neighbourhood of $C$.

\begin{lemma}
Let $v \in N(C)$. There are only two possibilities:
\begin{enumerate}[(i)]
\item $v$ has $2$ neighbours in $C$, $v_i$ and $v_{i+2}$ (we denote $v \in T_i$)
\item $v$ has $1$ neighbour in $C$, $v_i$ (we denote $v \in D_i$)
\end{enumerate}
for some $1 \leq i \leq k-2$ (see Figure \ref{picturek-2}).
\end{lemma}

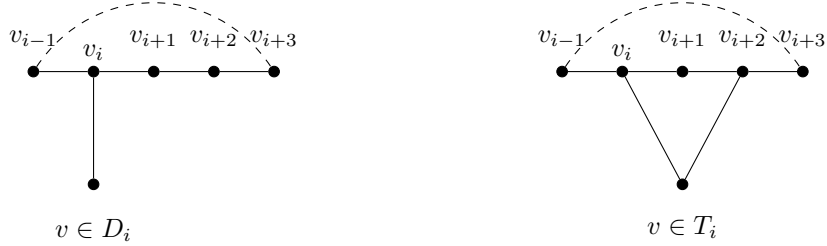
\begin{figure}
\centering
\begin{tikzpicture}
\begin{scope}
    \node (v1) [label=$v_{i-1}$] at (-1*0.8,0) {};
    \node (v2) [label=$v_{i  }$] at (0*0.8,0) {};
    \node (v3) [label=$v_{i+1}$] at (1*0.8,0) {};
    \node (v4) [label=$v_{i+2}$] at (2*0.8,0) {};
    \node (v5) [label=$v_{i+3}$] at (3*0.8,0) {};
    \node (v6) [label=below:$v \in D_i$] at (0*0.8,-1.5) {};
    \draw (v1) -- (v2);
    \draw (v2) -- (v3);
    \draw (v3) -- (v4);
    \draw (v4) -- (v5);
    \draw (v2) -- (v6);
    \draw [dashed] (v1) .. controls (0,1.2) and (2*0.8,1.2)  .. (v5);
\end{scope}
\begin{scope}[xshift=200]
    \node (v1) [label=$v_{i-1}$] at (-1*0.8,0) {};
    \node (v2) [label=$v_{i  }$] at (0*0.8,0) {};
    \node (v3) [label=$v_{i+1}$] at (1*0.8,0) {};
    \node (v4) [label=$v_{i+2}$] at (2*0.8,0) {};
    \node (v5) [label=$v_{i+3}$] at (3*0.8,0) {};
    \node (v6) [label=below:$v \in T_i$] at (1*0.8,-1.5) {};
    \draw (v1) -- (v2);
    \draw (v2) -- (v3);
    \draw (v3) -- (v4);
    \draw (v4) -- (v5);
    \draw (v2) -- (v6);
    \draw (v4) -- (v6);
    \draw [dashed] (v1) .. controls (0,1.2) and (2*0.8,1.2)  .. (v5);
\end{scope}
\end{tikzpicture}
\caption{Possible connections of $v \in N(C)$ to the cycle $C$ of length $k-2$. \label{picturek-2}}
\end{figure}

\begin{proof}
$v$ can have exactly one neighbour in $C$. 
$v$ cannot have two consecutive vertices of $C$ as neighbours since this would form a triangle.
Suppose $v$ is connected to some vertex $i$ and $i+t$ where $t$ is odd, $3 \leq t < k - 4$
and $v$ is not connected to $i+1, ..., i+t-1$. In this case we obtain an odd cycle of length smaller
than $k-2$. Notice that for every neighbour $v$ of $C$ there will be such $i$ and $t$ that
$v$ is connected to $i$ and $i+t$, $t$ is odd, $3 \leq t \leq k-4$ and $v$ is not connected to
$i+1,...,i+t-1$. So the only allowed value of $t$ is $k-4$. This means that neighbours of $v$
are in distance $(k-2) - (k-4) = 2$ in $C$.
\end{proof}

Let $D = \cup_{i=1}^{k-2} D_i$ and $T = \cup_{i=1}^{k-2} T_i$.
Let us denote by $S$ the set of vertices $S = C \cup D \cup T$.

\begin{lemma}
Let $v \in D_i$. $N(v) \subseteq C \cup T \cup D_{i\pm1} \cup D_{i\pm3}$. 
\end{lemma}

\begin{proof}
First of all $D_i$ is a stable set since otherwise we would have
a triangle. $v$ cannot be connected to $D_{i+2}$ since this would form $C_5$,
unless $k=7$ but then we have the case of $i+3$.
Vertices from $D_i$ cannot be connected to vertices of $D_{i+t}$ for
$4 \leq t \leq k - 6$ since this would give an odd cycle of length at most $k-4$.
\end{proof}

Let $S_1$ denote set of all vertices of type $T$, $C$, $D_i$ having edges to $D_{i\pm3}$ and all 
vertices of type $D$ influenced by neighbours of type $T$.

\begin{lemma} \label{colS1}
Either all vertices in $S_1$ have
exactly one possible color (determined by $C$ and at most $2$ other vertices)
or we cannot extend $C_{k-2}$-coloring of $C$ and there is a witness
consisting of $C$ and at most $4$ vertices.
\end{lemma}

\begin{proof}
We start from coloring vertices in $T_i$ by $i+1$.
Next we color vertices of type $D$ connected to vertices of type $T$ if they influence
their color.
Next we color vertices which are endpoints of edge connecting
$D_i$ to $D_{i+3}$ in the only possible way, their colors are $i+1$ and $i+2$.
If there was any conflict, we have a witness consisting of $C$ and at most $4$
other vertices.
Otherwise, color of each vertex is determined by $C$, himself and at most
1 other vertex.
\end{proof}

\begin{lemma} \label{colrest}
Either we can extend a $C_{k-2}$-coloring of $C \cup T$ to all components of $G-S$ having at least $2$
vertices or we have a witness of not $C_{k-2}$-colorability consisting of $C$ and at most 8 other
vertices.
\end{lemma}

\begin{proof}
Let $M$ denote a component of $G-S$ having at least $2$ vertices.
$M$ cannot have a neighbour of type $D$ because we would have a triangle or $P_k$.
Neighbours of this component in $S$ are of type $T$.
We again do similar structural analysis as in \cite{3col}. $M$ is bipartite
(otherwise we would have a cycle $C'$ of length $k-2$ in $G-S$ connected by some path $p_1..p_s$
to $C$, we would have induced $P_k$ $c'_3c'_2c'_1p_1..p_sc_i...c_{i+k-5}$). If $xyz$ is a path in $M$ then
$x$ and $z$ have the same neighbourhood in $S$ (otherwise let $u$ of type $T$ be a neighbour of $x$
not adjacent to $z$, $zyxuc_i...c_{i+k-5}$ is an induced $P_k$). So vertices in each of the two stable
sets in $M$ are adjacent to the same vertices in $S$. This means that we can color $M$ iff we can
color one edge of $M$ (all vertices in a stable set can be colored by the same color). We call this
edge $M'$.
Now take an induced subgraph $H$ of $G$. $H$ consists of $C$, $M'$
and one vertex $v$ from each $T_i$ for each vertex $u$ of $M'$ if $u$ is adjacent to $v$.
If a $C_{k-2}$-coloring of $C$ extends to a $C_{k-2}$-coloring of $H$ and we have colored $C \cup T$ then we can
color $C \cup T \cup M$. If it holds for every component $M$, we have colored $C$, $T$ and
all components of $G-S$ of size at least $2$. Otherwise,
the $C_{k-2}$-coloring of $C$ does not extend to $C_{k-2}$-coloring of some $H$. $H$ is a witness which
consists of $C$ and at most 8 other vertices (it is enough to take neighbours from 3 sets $T_i$).
\end{proof}

Let $S_2$ denote set of all vertices in $S_1$ and all 
vertices of type $D$ influenced by single vertex components in $G-S$.

\begin{lemma} \label{colS2}
Either all vertices in $S_2$ have exactly one possible color (determined by $C$ and at most $4$
other vertices) or we cannot extend the $C_{k-2}$-coloring of $S_1$ and there is a witness consisting of $C$ and at most
$8$ vertices. Moreover, the $C_{k-2}$-coloring of $S_2$ extends to $C_{k-2}$-coloring of $S$ iff the $C_{k-2}$-coloring of $S_2$ extends to
$C_{k-2}$-coloring of $S$ and single vertex components of $G-S$.
\end{lemma}

\begin{proof}
Let $v$ be a single vertex component of $G-S$.

If a vertex $v$ from $G-S$ has neighbours in $T_i$ and $T_j$ where
$i \neq j + 2$ and $j \neq i+2$, then $v$ cannot be colored. To ensure this we forbid some graphs
consisting of $C$ and 3 other vertices.

If $v$ has neighbours in $T_i$ and $T_{i+2}$, other neighbours of $v$ can belong only to
$D_i$, $D_{i+2}$ or $D_{i+4}$. We can forbid other configurations by graphs containing $C$
and 4 other vertices. Otherwise, we color $v$ by $i+2$, neighbours from $D_i$ by $i+1$,
neighbors from $D_{i+4}$ by $i+3$. Neighbuors from $D_{i+2}$ can have arbitrary color
from $i+1, i+3$. Color of such vertex in $D_i$ and $D_{i+4}$ results from $C$, himself
and 3 other vertices.

Further $v$ can have neighbours in $T_i$ only of type $T$,
other neighbours of $v$ can belong only to $D_{i-2}$, $D_i$, $D_{i+2}$, $D_{i+4}$.
We can forbid other configurations by graphs containing $C$ and 3 other vertices.
$v$ cannot have neighbours both in $D_{i-2}$ and $D_{i+4}$. We can forbid it by a graph consisting
of $C$ and 4 other vertices. If $v$ has neighbours in one of these sets, say $D_{i-2}$,
then we color $v$ to $i$, neighbours in $D_{i-2}$ to $i-1$, neighbours in $D_{i+2}$ to
$i+1$. Neighbours from $D_i$ can have arbitrary color from $i-1,i+1$. Color of such vertex
in $D_{i-2}, D_{i+2}$ results from $C$, himself and 3 other vertices.
The other case is symmetric. Suppose $v$ has neighbours only in $D_i$ and $D_{i+2}$.
$v$ cannot have neighbours in both of these sets, since we would get $C_k$. If $v$ has neighbours
only in $D_i$ we can always color $v$ by $i$. The other case is symmetric.

Further $v$ can have neighbours of type $D$ only. $v$ can have only neighbours
in $D_i$ and $D_{i+4}$ for some $i$. $v$ cannot have neighbours in $D_{i+1}$ because we would
have $C_5$, and in case of $k=7$ this corresponds to case $i+4$. $v$ cannot have neighbours
in $D_{i+2}$ since this would give $C_k$. $v$ cannot have neighbours in $D_{i+3}$ since
we would have $C_7$, and in case of $k=9$ this corresponds to case $i+4$. $v$ cannot have neighbours
in $D_t$ for $5 \leq t$ since we would not be able to color $v$ any more. We can forbid
this situation by graphs consisting of $C$ and 3 other vertices. In case of $v$ connected
to $D_i$ and $D_{i+4}$ we color $v$ to $i+2$, neighbors of $v$ in $D_i$ to $i+1$, neighbours
of $v$ in $D_{i+4}$ to $i+3$. Color of such vertex of type $D$ results from $C$, himself
and 2 other vertices. If $v$ has neighbours only of type $D_i$ we can always color $v$ to $i$.

When applying the rules above when coloring $v$ we could have
a conflict or not. If we have a conflict, color of some vertex $v$ of type $D_i$ can be a problem,
then we have a witness which is a sum of graphs forcing color of $v$. Such graph consists
of $C$, $v$ and at most 6 other vertices. Conflict can happen on some edge between two vertices
of type $D$. In the same way we have a witness consisting of $C$ and at most 8 other vertices.
\end{proof}

To sum up our current state. We either found a forbidden graph consisting of $C$ and at most 8
other vertices or we have a partial $C_{k-2}$-coloring. If there exists any $C_{k-2}$-coloring of $G$ then there
exists a $C_{k-2}$-coloring of $G$ extending our $C_{k-2}$-coloring. We have colored vertices of $C$, all vertices
of type $T$, all vertices of $G-S$, all vertices of type $D_i$ having connection to vertex
of type $D$ other than $D_{i-1}$ or $D_{i+1}$, all vertices of type $D$ that can be influenced
by colors in $G-S$, all vertices of type $D$ that can be influenced by vertices of type $T$.
This means that we have left to color only some vertices of type $D$ and the only edges we need
to care for connect $D_i$ to $D_{i+1}$.

\begin{lemma} \label{colS}
Either we can extend the $C_{k-2}$-coloring of $S_2$ to $C_{k-2}$-coloring of $S$ or there is a witness
consisting of $C$ and at most $2k+4$ other vertices in $G$.
\end{lemma}

\begin{proof}
Graph induced by vertices of type $D$ contains no induced $C_{k-2}$. If it contained one,
it would have one vertex from each set $D_i$. By removing one of these vertices and taking 3
vertices of $C$ we would get $C_k$.

Each vertex of $D_i$ has two possible colors $i-1$ and $i+1$. We will call them "small" and "big".
If $v$ uses the small color and $u \in D_{i+1}$ then $u$ also has to use
the small color. If $u$ uses the big color then $v$ also has to use the big color.
In one step procedure we can color all vertices that have neighbours as above.
If we do not have any conflicts we will end in at most $k-2$ steps. From one vertex
we can expand only in one direction. Starting from $a \in D_i$ after $k-2$ steps we
get to $b \in D_i$ ($a \neq b$ since we have no $C_{k-2}$). We have induced $P_{k-1}$
since there are no crossing edges among not colored vertices. Suppose we make one
more step to $c \in D_{i+1}$. Let $d \in D_{i+1}$ be the neighbour of $a$ on
our path and $e \in D_{i+2}$ be the neighbour of $d$.
We cannot have edge $(a,c)$,$(b,d)$ or $(c,e)$ because we would have $C_k$ or $C_{k-2}$,
so we obtain $P_k$.

Suppose this procedure is finished without conflicts after at most $k-2$ steps.
We can color all remaining vertices for example by small color and obtain coloring of $G$.
Otherwise, we get some conflict in vertex $v$ or edge $(u,v)$. We pick
vertices $a$ and $b$ from which we started paths leading to a conflict. Colors of $a$ and $b$
are defined by graphs consisting of $C$ and at most 4 other vertices each. Together with two paths
of length at most $k-2$ they form our witness. This witness consists of $C$ and at most
$2(k-2) + 8 = 2k + 4$ other vertices, altogether $3k + 2$ vertices. 
\end{proof}

\begin{lemma} \label{Ck-2}
Either we can $C_{k-2}$-color $G$ or $G$ contains a not $C_{k-2}$-colorable induced subgraph
of size at most $3k + 2$ containing $C_{k-2}$.
\end{lemma}

\begin{proof}
This follows form Lemmas \ref{colS1}, \ref{colrest}, \ref{colS2}, \ref{colS}
and the fact that components of $G-S$ of size at least 2 do not have neighbours of
type $D$.
\end{proof}

\end{subsection}

\begin{paragraph}{Algorithm summary.}
The coloring algorithm from Theorem \ref{main2} can be easily deduced from the proof of Theorem \ref{main1}. Let us summarize its steps. 

First we check if $G$ is bipartite and if so we color $G$ by 2 subsequent vertices of $C_{k-2}$. If $G$ is not bipartite and $K$ is even, we return an odd cycle of length at most $k-1$ as a witness that $G$ is not $C_{k-2}$-colorable. Otherwise, we check if $G$ contains an odd cycle of length
at most $k-4$. If so, $G$ is not $C_{k-2}$-colorable and we return the cycle as a witness. Now, depending on whether $G$ contains a $C_k$ or is $C_k$-free, we have two sets of forbidden subgraphs. In any case, the forbidden graphs have at most $3k+28$ vertices ($3k+2$ or $k+42$ from Lemmas \ref{Ck} and \ref{Ck-2}). If $G$ contains one of them as induced subgraph, $G$ is not $C_{k-2}$-colorable and the subgraph is the witness. Otherwise, we can $C_{k-2}$-color $G$ following the proof.
\end{paragraph} 
\end{section}

\bigskip
\noindent We would like to conclude the paper with two open problems. \medskip

\noindent (1) What is the complexity of $C_{k-2}$-list-coloring of $P_k$-free graphs? 

\noindent (2) What is the complexity of $H$-coloring $P_k$-free graphs where $H$ is a cycle shorter than
$C_{k-2}$? A more general question would be to determine the complexity of $H$-coloring in other
restricted graph classes.

\end{document}